\documentclass[10pt,reqno]{amsart}
\usepackage{a4wide}
\usepackage[bbgreekl]{mathbbol}
\usepackage[mathscr]{eucal}
\usepackage{amssymb}
\usepackage[all]{xy}

\usepackage[usenames]{color}

\newtheorem{prop}{Proposition}[section]

\newtheorem{cor}[prop]{Corollary}

\theoremstyle{remark}
\newtheorem{rem}[prop]{Remark}
\numberwithin{equation}{section}
\numberwithin{prop}{section}

\newcommand{\cst}{\mathrm{C^*}}
\newcommand{\tens}{\otimes}
\newcommand{\id}{\mathrm{id}}
\newcommand{\CC}{\mathbb{C}}
\newcommand{\RR}{\mathbb{R}}
\newcommand{\TT}{\mathbb{T}}
\newcommand{\ZZ}{\mathbb{Z}}
\newcommand{\hh}{\mathbb{h}}
\newcommand{\I}{\mathbb{1}}
\newcommand{\cF}{\mathcal{F}}

\newcommand{\cB}{\mathcal{B}}
\newcommand{\cA}{\mathcal{A}}
\newcommand{\comp}{\!\circ\!}
\newcommand{\vt}{\!\vartriangle\!}
\newcommand{\cJ}{\mathcal{J}}
\newcommand{\st}{\:\vline\:}
\newcommand{\Mod}{\;\mathrm{mod}\;}
\newcommand{\malpha}{\mathbb{\bbalpha}}
\newcommand{\mbeta}{\mathbb{\bbbeta}}
\newcommand{\mgamma}{\mathbb{\bbgamma}}
\newcommand{\mdelta}{\mathbb{\bbdelta}}
\newcommand{\AM}{\mathbb{A}}
\newcommand{\DM}{\mathbb{\Delta}}
\newcommand{\PM}{\mathbb{\Phi}}
\newcommand{\eM}{\mathbb{\bbespilon}}
\DeclareMathOperator{\C}{C}
\DeclareMathOperator{\Mor}{Mor}
\DeclareMathOperator{\M}{M}

\DeclareMathOperator{\Aut}{Aut}
\DeclareMathOperator{\QMap}{\mathscr{Q}-Map}
\DeclareMathOperator{\qs}{\mathscr{QS}}

\begin{document}

\title{Examples of quantum commutants}

\date{11 August, 2008}

\author{Piotr M.~So{\l}tan}
\address{Department of Mathematical Methods in Physics\\
Faculty of Physics\\
Warsaw University}
\email{piotr.soltan@fuw.edu.pl}

\thanks{Research partially supported by Polish government grants
no.~115/E-343/SPB/6.PRUE/DIE50/2005-2008 and N201 1770 33.}

\maketitle

\begin{abstract}
We describe the notion of a quantum family of maps of a quantum space and that of a quantum commutant of such a family. Quantum commutants are quantum semigroups defined by a certain universal property. We give a few examples of these objects acting on a classical $n$-point space and on the quantum space underlying the algebra of two by two matrices. We show that some of the resulting quantum semigroups are not compact quantum groups. The proof of one result touches on an interesting problem of the theory of compact quantum groups.
\end{abstract}

\section{Introduction}

The notion of a quantum space is by now a fairly well established concept {(\cite{ncg,pseu,pseudogr,podles,cqg})}. A quantum space is an object of the category dual to the category of $\cst$-algebras (\cite{pseu,gen}). If $B$ is a $\cst$-algebra then by $\qs(B)$ we will denote the ``quantum space underlying $B$''. In \cite{qs} we defined and studied \emph{quantum families of maps} which are families of maps of quantum spaces parameterized by some other quantum space (cf.~Section \ref{i2}).

In this paper we will give a few examples of the objects also introduced in \cite{qs} called \emph{quantum commutants.} These are certain compact quantum semigroups defined as universal quantum families of maps of some quantum space into itself commuting with a given family of maps. In the course of analysis of our examples, some of them will turn out to be classical objects (and in fact classical finite groups) while others will fail to be compact quantum groups.

{Despite a superficial similarity, the quantum semigroups we shall consider are \emph{not} coalgebras (except the one found in Example \eqref{examples1}). Also their actions on quantum spaces (objects dual to $\cst$-algebras) are \emph{not} coactions. Similarly compact quantum groups are \emph{not} Hopf algebras. All tensor product of $\cst$-algebras are completed in the minimal tensor product norm. For this reason we shall not use the terminology of coalgebra theory.}

Let us briefly describe the contents of the paper. In Section \ref{i2} we recall some basic notions from \cite{qs} such as quantum families of maps and quantum commutant. Section \ref{Xn} deals with the special situation when we study quantum commutants of classical families of permutations of finite sets. A general presentation of such an object is given and some examples are analyzed. In Section \ref{M2} we first describe in detail the quantum semigroup $\QMap\bigl(\qs(M_2)\bigr)$ and find explicitely the quantum commutant of a family consisting of a single automorphism $\phi$ of $M_2$. For a particular choice of $\phi$ we show that the resulting compact quantum semigroup is not a compact quantum group. To that end we use a wide array of results about compact quantum groups. Our analysis also shows that in this case the quantum commutant is different from the classical one which is a group isomorphic to $\TT\rtimes\ZZ_2$.

\section{Quantum families of maps and quantum commutants}\label{i2}

\subsection{Quantum families of maps}
\noindent
{Let $A$ and $B$ be $\cst$-algebras. By a \emph{morphism} from $A$ to $B$ we understand a nondegenerate $*$-homomorphism $A\to\M(B)$ (the multiplier algebra of $B$). The category of quantum spaces is the category dual to the category whose objects are $\cst$-algebras with these morphisms. The set of all morphisms from $A\to{B}$ is denoted by $\Mor(A,B)$. (cf.~\cite{pseu,unbo,gen}). In the remainder of the paper we shall only encounter unital $\cst$-algebras, so the morphisms will be simply unital $*$-homomorphisms between $\cst$-algebras. However, the universal properties of some of the considered quantum families of maps (in particular the quantum commutants) hold for the larger class of morphisms we just described.

Now let $M$ and $B$ be $\cst$-algebras. }
A \emph{quantum family of maps $\qs(M)\to\qs(M)$} labeled by $\qs(B)$ is an element $\Psi_B\in\Mor(M,M\tens{B})$. In \cite{qs} we showed that if $M$ is finite dimensional then there exists a $\cst$-algebra $\AM$ and a distinguished quantum family $\PM\in\Mor(M,M\tens\AM)$ such that for any $B$ and $\Psi_B\in\Mor(M,M\tens{B})$ there exists a unique $\Lambda\in\Mor(\AM,B)$ such that the diagram
\[
\xymatrix{
M\ar[rr]^-{\PM}\ar@{=}[d]&&M\tens\AM\ar[d]^{\id_M\tens\Lambda}\\
M\ar[rr]^-{\Psi_B}&&M\tens{B}}
\]
is commutative. The quantum family $\PM$ was then called the \emph{quantum family of all maps $\qs(M)\to\qs(M)$.} It was shown that there is a comultiplication $\DM$ on $\AM$ making $(\AM,\DM)$ into a compact quantum semigroup with unit {(in other words $\AM$ is a unital $\cst$-algebra and $\DM\in\Mor(\AM,\AM\tens\AM)$ is a coassociative morphism; moreover there exists on $\AM$ a continuous counit $\eM\in\Mor(A,\CC)$)}. This {quantum} semigroup was denoted by $\QMap\bigl(\qs(M)\bigr)$.

In \cite{qs} we investigated the object $\QMap\bigl(\qs(M)\bigr)$ and its subobjects. In particular we defined quantum semigroups preserving a fixed state on $M$ (\cite[Section 5]{qs}). One can use those results to prove some other general properties of $\QMap\bigl(\qs(M)\bigr)$.

\begin{prop}
If $M$ is more than one dimensional then $\QMap\bigl(\qs(M)\bigr)$ is not a compact quantum group. 
\end{prop}

\begin{proof}
If $\QMap\bigl(\qs(M)\bigr)$ were a compact quantum group, it would have the Haar measure $\hh$. The standard formula (cf.~\cite{boca,podles})
\[
E:M\ni{m}\longmapsto(\id_M\tens\hh)\PM(m)\in{M}
\]\sloppy
yields a conditional expectation whose range consist of elements invariant for the action of $\QMap\bigl(\qs(M)\bigr)$ on $\qs(M)$:
\[
\begin{split}
\PM\bigl(E(m)\bigr)&=\PM\bigl((\id_M\tens\hh)\PM(m)\bigr)\\
&=(\id_M\tens\id_\AM\tens\hh)\bigl((\PM\tens\id_\AM)\PM(m)\bigr)\\
&=(\id_M\tens\id_\AM\tens\hh)\bigl((\id_M\tens\DM)\PM(m)\bigr)\\
&=\bigl(\id_M\tens[(\id_\AM\tens\hh)\comp\DM]\bigr)\PM(m)\\
&=\bigl[(\id_M\tens\hh)\PM(m)\bigr]\tens\I_\AM=E(m)\tens\I_\AM.
\end{split}
\]
We have shown, however, in \cite[Proposition 4.6]{qs} that $\QMap\bigl(\qs(M)\bigr)$ acts ergodically on $\qs(M)$. In other words $E(m)$ must be a multiple of $\I_M$. Therefore the formula
\[
E(m)=\omega(m)\I_M
\]
defines an invariant state on $M$. Indeed, for any $m\in{M}$ we have
\[
\begin{split}
\I_M\tens\bigl[(\omega\tens\id_\AM)\PM(m)\bigr]&=(E\tens\id_\AM)\PM(m)\\
&=(\id_M\tens\hh\tens\id_\AM)(\PM\tens\id_\AM)\PM(m)\\
&=(\id_M\tens\hh\tens\id_\AM)(\id_M\tens\DM)\PM(m)\\
&=\bigl[(\id_M\tens\hh)\PM(m)\bigr]\tens\I_\AM\\
&=E(m)\tens\I_\AM=\omega(m)\I_M\tens\I_\AM=\I_M\tens\omega(m)\I_\AM.
\end{split}
\]
Therefore $(\omega\tens\id_\AM)\PM(m)=\omega(m)\I_\AM$ for all $m\in{M}$. Now \cite[Proposition 5.3]{qs} says that there is no invariant state on $M$ unless $\dim{M}=1$.
\end{proof}

Let $M$ be a $\cst$-algebra and let $\Psi_B\in\Mor(M,M\tens{B})$ and $\Psi_C\in\Mor(M,M\tens{C})$ be quantum families of maps $\qs(M)\to\qs(M)$. The \emph{composition} of the quantum families $\Psi_B$ and $\Psi_C$ is the quantum family labeled by $\qs(B\tens{C})$ defined as $(\Psi_B\tens\id_C)\comp\Psi_C$. The composition of $\Psi_B$ and $\Psi_C$ is denoted by $\Psi_B\vt\Psi_C$ (\cite[Section 3]{qs}).

\subsection{Quantum commutants}

Let $B$ be a $\cst$-algebra and let $\Psi_B\in\Mor(M,M\tens{B})$ be a quantum family of maps. The quantum commutant $\QMap_{\Psi_B}\bigl(\qs(M)\bigr)$ of $\Psi_B$ is the quantum semigroup $(A,\Delta)$ where the quantum space $\qs(A)$ labels the universal quantum family $\Phi\in\Mor(M,M\tens{A})$ of maps commuting with $\Psi_B$. The notion of commuting families was introduced in \cite[Section 6]{qs} and is a generalization of a classical notion of commuting families of maps.

Let us recall some of the features of  quantum commutants proved in \cite{qs}:
\begin{itemize}
\item the $\cst$-algebra $A$ is the quotient $\AM/\cJ$, where $\cJ$ is the ideal in $\AM$ generated by the set
\begin{equation}\label{K}
\Bigl\{(\omega\tens\id_\AM\tens\eta)(\PM\vt\Psi_B)(m)
-(\omega\tens\eta\tens\id_\AM)(\Psi_B\vt\PM)(m)
\st{m}\in{M},\:\omega\in{M^*},\:\eta\in{B^*}\Bigr\},
\end{equation}
\item The quantum family $\Phi\in\Mor(M,M\tens{A})$ satisfies $\Phi=(\id_M\tens\pi)\PM$, where $\pi$ is the quotient map $\AM\to{A}$. Moreover $\Phi$ is an action of $\QMap_{\Psi_M}\bigl(\qs(M)\bigr)$ on $\qs(M)$.
\item $\QMap_{\Psi_B}\bigl(\qs(M)\bigr)$ is a quantum subsemigroup with unit of $\QMap\bigl(\qs(M)\bigr)$, i.e.~the quotient map $\pi$ is a quantum semigroup morphism and {
\begin{equation}\label{epie}
\varepsilon\comp\pi=\eM,
\end{equation}
}
where {$\varepsilon$} is the counit of $\QMap_{\Psi_B}\bigl(\qs(M)\bigr)$.
\end{itemize}
In order to illustrate the notion of a quantum commutant further let us consider the following simple example: a quantum family $\Psi_B\in\Mor(M,M\tens{B})$ of maps $\qs(M)\to\qs(M)$ is \emph{trivial} if $\Psi_B(m)=m\tens\I_B$ for all $m\in{M}$. In \cite[Proposition 6.1]{qs} we showed that any quantum family of maps commutes with a trivial family. Therefore (by universality) the quantum commutant of a trivial family must be the whole quantum semigroup $\QMap\bigl(\qs(M)\bigr)$.

\subsubsection*{Notational convention}

In each of the next two sections we will fix a finite dimensional $\cst$-algebra $M$ and first consider the quantum semigroup $\QMap\bigl(\qs(M)\bigr)$. In each case we will denote the corresponding $\cst$-algebra with comultiplication by $(\AM,\DM)$, its counit by $\eM$ and the associated action on $\qs(M)$ by $\PM$. Then we will study quantum commutants of some fixed family of maps $\qs(M)\to\qs(M)$ usually denoted by $\Psi_B\in\Mor(M,M\tens{B})$ with some $\cst$-algebra $B$. The $\cst$-algebra with comultiplication describing the quantum commutant $\QMap_{\Psi_B}\bigl(\qs(M)\bigr)$ will then be denoted by $(A,\Delta)$ and its counit will always be denoted by {$\varepsilon$}, while its action on $\qs(M)$ will be denoted by $\Phi\in\Mor(M,M\tens{A})$. {In other words the passage from $\QMap\bigl(\qs(M)\bigr)$ to $\QMap_{\Psi_B}\bigl(\qs(M)\bigr)$ will be symbolically encoded in the following transformation of notation:
\[
\xymatrix{\bigl(\AM,\DM,\PM,\eM\bigr)\ar@{|~>}[r]&\bigl(A,\Delta,\Phi,\varepsilon\bigr).}
\]
}

\section{Quantum commutants of classical families of bijections}\label{Xn}

Let $X_n$ denote an $n$ element set. This finite space is described by the commutative $\cst$-algebra $\C(X_n)=\CC^n$. In other words $X_n=\qs(\CC^n)$. In this section we shall give a description of quantum commutants of classical families of bijections $X_n\to{X_n}$.

\begin{prop}\label{propWangR}
Let $M=\CC^n$. Then $\QMap(X_n)=(\AM,\DM)$, where $\AM$ is the universal $\cst$-algebra generated by elements $\{a_{i,j}\st{1\leq{i,j}}\leq{n}\}$ with relations
\begin{equation}\label{WangR}
\begin{aligned}
(a_{ij})^*(a_{ij})&=a_{ij},&i,j&=1,\ldots,n,\\
\sum_{j=1}^na_{ij}&=\I_\AM,&i&=1,\ldots,n
\end{aligned}
\end{equation}
and comultiplication $\DM\in\Mor(\AM,\AM\tens\AM)$
\begin{align}
\DM(a_{i,j})&=\sum_{k=1}^na_{i,k}\tens{a_{k,j}},&i,j&=1,\ldots,n,\nonumber
\end{align}
while the counit $\eM$ maps $a_{i,j}$ to $1$ if $i=j$ and to $0$ otherwise.

The action $\PM\in\Mor(M,M\tens\AM)$ of $\QMap(X_n)$ on $X_n$ is given on the standard basis $\{e_1,\ldots,e_n\}$ of $M=\CC^n$ by
\[
\PM(e_j)=\sum_{i=1}^ne_i\tens{a_{i,j}}.
\]
\end{prop}

Proposition \ref{propWangR} is simple to prove and is implicitely contained in \cite[Theorem 3.1 \& remark (3) on page 208]{wang}. In section \ref{M2} we will give a proof of a similar result (Proposition \ref{factsA}).

Let $\cF$ be a classical family of maps $X_n\to{X_n}$. We will only consider finite families, so that $\cF$ consists of a finite number, say $m$, elements and we can view $\cF$ as a morphism $\Psi_B\in\Mor(M,M\tens{B})$, where $B=\C(\cF)=\CC^m$ and
\[
\Psi_B(e_j)=\sum_{\sigma\in\cF}e_{\sigma(j)}\tens\delta_\sigma,
\]
where for each $\sigma\in\cF$ the symbol $\delta_\sigma$ denotes a function on the set $\cF$ equal to $1$ at $\sigma$ and zero in all other points. We shall now describe the construction of a quantum commutant of the family $\Psi_B$. With a small abuse of terminology we shall call this quantum semigroup the quantum commutant of $\cF$ and denote it by $\QMap_\cF(X_n)$. The corresponding $\cst$-algebra with comultiplication will be denoted by $(A,\Delta)$.

\begin{prop}\label{sigsig}
Let $M=\C(X_n)$ and let $\cF$ be a classical family of bijections $X_n\to{X_n}$. Let $B=\CC^{|\cF|}$ and let $\Psi_B\in\Mor(M,M\tens{B})$ correspond to the family $\cF$ as described above. Then the ideal \eqref{K} of $\AM$ is generated by the elements
\[
\bigl\{a_{i,\sigma(j)}-a_{\sigma^{-1}(i),j}\st\:1\leq{i,j}\leq{n},\:\sigma\in\cF\bigr\}.
\]
Consequently $A$ is generated by elements $\{a_{ij}\st\;i,j=1,\ldots,n\}$ satisfying \eqref{WangR} and
\begin{equation}\label{ss}
a_{\sigma(i),\sigma(j)}=a_{i,j},\qquad{i,j=1,\ldots,n,}\quad\sigma\in\cF.
\end{equation}
Moreover the commutant of $\cF$ coincides with the commutant of the group generated by $\cF$.
\end{prop}

\begin{proof}
In order to find a set of generators of the ideal \eqref{K} it is enough to take for $m$, $\omega$ and $\eta$ elements of fixed bases of $M$, $M^*$ and $B^*$ respectively. Fix $i\in\{1,\ldots,n\}$ and $\sigma\in\cF$ and let $\omega$ be equal to $1$ on $e_i$ and to $0$ on other vectors of the standard basis of $M$. Similarly let $\eta$ pick out $\delta_\sigma$ and kill all other $\delta_\tau$ with $\tau\in\cF$. We have
\[
(\PM\vt\Psi_B)(e_j)=(\PM\tens\id_B)\biggl(\sum_{\tau\in\cF}e_{\tau(j)}\tens\delta_\tau\biggr)=
\sum_{k=1}^n\sum_{\tau\in\cF}e_k\tens{a_{k,\tau(j)}}\tens\delta_\tau
\]
and
\[
(\Psi_B\vt\PM)(e_j)=(\Psi_B\tens\id_\AM)\biggl(\sum_{k=1}^ne_k\tens{a_{k,j}}\biggr)=
\sum_{\tau\in\cF}\sum_{k=1}^ne_{\tau(k)}\tens\delta_\tau\tens{a_{k,j}}.
\]
Using the fact that all elements of $\cF$ are bijections we see that
\[
(\omega\tens\id_\AM\tens\eta)(\PM\vt\Psi_B)(e_j)
-(\omega\tens\eta\tens\id_\AM)(\Psi_B\vt\PM)(e_j)
=a_{i,\sigma(j)}-a_{\sigma^{-1}(i),j}.
\]
Substituting $\sigma(i)$ in place of $i$ we  obtain the relation \eqref{ss} for generators of $A$. Moreover We see that if $\sigma_1,\sigma,_2\in\cF$, then
\[
R_{i,j}=a_{\sigma_1(i),\sigma_1(j)}-a_{i,j}\quad\text{and}\quad
S_{i,j}=a_{\sigma_2(i),\sigma_2(j)}-a_{i,j}
\]
are zero in $A$ for all $i,j$. Therefore
\[
a_{\sigma_1(\sigma_2(i)),\sigma_1(\sigma_2(j))}-a_{i,j}=
R_{\sigma_2(i),\sigma_2(j)}-S_{i,j}
\]
is also equal to $0$ in $A$. It follows that taking compositions of elements of $\cF$ does not enlarge the commutant. One could follow a similar reasoning with inverses, but it is enough to note that inverses in a finite group (in our case contained in the symmetric group $S_n$) are expressible as products (namely a high enough power of a given element is its inverse).
\end{proof}

\subsubsection{Example}\label{examples1}
For $n\geq{3}$ let us consider the family of maps $X_n\to{X_n}$ consisting of a single cyclic permutation $i\mapsto{i+1}\Mod{n}$. The quantum commutant $\QMap_\cF(X_n)$ is then isomorphic to the classical group $\ZZ_n$ acting on $X_n$ by cyclic permutations. Indeed, let $\pi:\AM\to{A}$ be the quotient map. Then $A$ is generated by images of generators $a_{i,j}$ of $\AM$. For $k=1,\ldots,n$ let $x_k=\pi(a_{1,k})$. Then we easily see that $x_1,\ldots,x_n$ generate $A$ and they are self-adjoint projections adding up to $\I_A$. Therefore they commute and we see that $A$ is isomorphic to $\CC_n$. Moreover, since $\pi$ preserves comultiplication, we find that
\[
\Delta(x_k)=\sum_{p=1}^nx_p\tens{x_{k-p+1\Mod{n}}}
\]
One easily sees that $(A,\Delta)$ is isomorphic to the algebra of all functions on $\ZZ_n$ with $x_k$ mapped to the ``delta function'' at $k-1\in\ZZ_n$. In particular this quantum commutant is not only a quantum group, but even a classical (and commutative) finite group. One can show by direct calculation that for $n=2$ the quantum commutant of a cyclic permutation (which in this case is also a transposition) is the trivial group.

\subsubsection{Example}
\noindent{{
Let us now specify the situation from Example \ref{examples1}. Let $\cF$ be a family consisting of a single map $\sigma:X_n\to{X_n}$ exchanging two points (a transposition). By putting $X_n=\{1,\ldots,n\}$ we can assume that $\sigma=(n-1,n)$. We shall describe the quantum commutant $\QMap_\cF(X_n)=(A,\Delta)$ of $\cF$.

We have the quotient map $\pi:\AM\to{A}$, so images under $\pi$ of generators $(a_{i,j})_{i,j=1,\ldots,n}$ of $\AM$ (cf.~Proposition \ref{propWangR}) generate $A$. Let us denote the image of the matrix $(a_{i,j})_{i,j=1,\ldots,n}\in{M_n\tens\AM}$ under $\id_{M_n}\tens\pi$ by
\[
\begin{bmatrix}
b_{1,1}&\cdots&b_{1,n-2}&d_1&d_1'\\
\vdots&\ddots&\vdots&\vdots\\
b_{n-2,1}&\cdots&b_{n-2,n-2}&d_{n-2}&d_{n-2}'\\
c_1&\cdots&c_{n-2}&e&f\\
c_1'&\cdots&c_{n-2}'&f'&e'
\end{bmatrix}.
\]
First let us remark that it follows from Proposition \ref{sigsig} that 
$e=e'$, $f=f'$ and $c_j=c_j'$, $d_j=d_j'$ for $j=1,\ldots,n-2$. Moreover the defining relations of $\AM$ imply that $d_j=0$ for all $j$. The action $\Phi\in\Mor(\CC^n,\CC^n\tens{A})$ is given on generators by
\[
\Phi(e_j)=\sum_{i=1}^{n-2}e_i\tens{b_{i,j}}+(e_{n-1}+e_n)\tens{c_j}
\]
for $1\leq{j}\leq{n-2}$ and
\[
\begin{split}
\Phi(e_{n-1})&=e_{n-1}\tens{e}+e_n\tens{f},\\
\Phi(e_n)&=e_{n-1}\tens{f}+e_n\tens{e}.
\end{split}
\]

Let $B$ be the $\cst$-subalgebra of $A$ generated by $(b_{i,j})_{i,j=1,\ldots,n-2}$ and $C$ be the $\cst$-subalgebra generated by $e,f,c_1,\ldots,c_{n-2}$. Clearly $C$ is isomorphic to $\CC^n$. Note also that $\bigl.\Delta\bigr|_B\in\Mor(B,B\tens{B})$. We will show that
\begin{enumerate}
\item\label{pa} $(B,\Delta)$ is isomorphic to the quantum semigroup $\QMap(X_{n-2})$,
\item\label{pb} $A\cong{B*C}$.
\end{enumerate}
To see point \eqref{pa} let $D$ be a $\cst$-algebra and let $\Psi_D\in\Mor(\CC^{n-2},\CC^{n-2}\tens{D})$ be a quantum family of maps $X_{n-2}\to{X_{n-2}}$. Define
$\widetilde{\Psi}_D\in\Mor(\CC^n,\CC^n\tens{D})$ by 
\[
\widetilde{\Psi}_D(e_j)=(\iota\tens\id_D)\Psi_D(e_j)
\]
(where $\iota$ is the inclusion of $\CC^{n-2}$ into $\CC^n$ onto the subspace spanned by the first $n-2$ vectors of the standard basis) for $j=1,\ldots,n-2$ and
\[
\begin{split}
\widetilde{\Psi}_D(e_k)=e_k\tens\I_D
\end{split}
\]
for $k=n-1,n$. Then $\widetilde{\Psi}_D$ is a quantum family of maps $X_n\to{X_n}$ commuting with $\cF$. The universal property of $\QMap_\cF(X_n)$ shows that there exists a unique map $\Lambda\in\Mor(A,D)$ such that $(\id_{\CC^n}\tens\Lambda)\comp\Phi=\widetilde{\Psi}_D$.

Let now $p:\CC^n\to\CC^{n-2}$ be the projection onto the first $n-2$ coordinates and let $\Phi'=(p\tens\id_A)\comp\Phi\comp\iota:\CC^{n-2}\to\CC^{n-2}\tens{B}$. As defined here $\Phi'$ is a completely positive map, but one immediately sees that it is in fact a unital $*$-homomorphism.

Evidently $\Lambda'=\bigl.\Lambda\bigr|_{B}$ is now a morphism from $B$ to $D$ such that
$\Psi_D=(\id_B\tens\Lambda')\comp\Phi'$. Moreover this morphism is unique (if there were different ones we could easily extend them both to $A$ and have different choices for $\Lambda$). It follows that $(B,\Phi')$ has the universal property of $\QMap(X_{n-2})$.

To prove \eqref{pb} one first uses the universal property of the free product to construct a map $B*C\to{A}$. Namely, let 
\begin{equation}\label{ge}
\bigl(\widetilde{b}_{i,j}\bigr)_{i,j=1,\ldots,n-2}\quad\text{and}
\quad\widetilde{e},\ \widetilde{f},\  \widetilde{c}_1,\ldots,\widetilde{c}_{n-2}
\end{equation}
be images of the generators of $B$ and $C$ in $B*C$. Then there is a morphism $\Theta\in\Mor(B*C,A)$ sending each of the generators \eqref{ge} to the corresponding generator in $A$. To define a map in the opposite direction let $\widetilde{\Phi}\in\Mor\bigl(\CC^n,\CC^n\tens(B*C)\bigr)$ be given by
\[
\widetilde{\Phi}(e_j)=\sum_{i=1}^{n-2}e_i\tens{\widetilde{b}_{i,j}}
+(e_{n-1}+e_n)\tens{\widetilde{c}_j}
\]
for $1\leq{j}\leq{n-2}$ and
\[
\begin{split}
\widetilde{\Phi}(e_{n-1})&=e_{n-1}\tens\widetilde{e}+e_n\tens\widetilde{f},\\
\widetilde{\Phi}(e_n)&=e_{n-1}\tens\widetilde{f}+e_n\tens\widetilde{e}
\end{split}
\]
(one has to check that this map exists, i.e.~the images of the basis vectors are projections summing up to $\I$). Then $\widetilde{\Phi}$ is a quantum family of maps $X_n\to{X_n}$ commuting with $\cF$ and the universal property of $\QMap_\cF(X_n)$ provides an element of $\Lambda\in\Mor(A,B*C)$ such that
$(\id_{\CC^n}\tens\Lambda)\comp\Phi=\widetilde{\Phi}$. By looking at the images of generators one can check that $\Lambda$ is then the inverse of $\Theta$.
}}

\section{An example on $M_2$}\label{M2}

In this section we shall consider the finite quantum space $\qs(M_2)$. We begin by describing the quantum space of all maps $\qs(M_2)\to\qs(M_2)$.

\begin{prop}\label{factsA}
The quantum space $\QMap\bigl(\qs(M)\bigr)=(\AM,\DM)$, where $\AM$ is the universal $\cst$-algebra generated by four elements $\malpha,\mbeta,\mgamma$ and $\mdelta$ satisfying the relations:
\begin{equation}\label{nstn}
\begin{split}
\malpha^*\malpha+\mgamma^*\mgamma+\malpha\malpha^*+\mbeta\mbeta^*&=\I,\\
\malpha^*\mbeta+\mgamma^*\mdelta+\malpha\mgamma^*+\mbeta\mdelta^*&=0,\\
\mbeta^*\mbeta+\mdelta^*\mdelta+\mgamma\mgamma^*+\mdelta\mdelta^*&=\I
\end{split}
\end{equation}
and
\begin{equation}\label{nkw}
\begin{split}
\malpha^2+\mbeta\mgamma&=0,\\
\malpha\mbeta+\mbeta\mdelta&=0,\\
\mgamma\malpha+\mdelta\mgamma&=0,\\
\mgamma\mbeta+\mdelta^2&=0.
\end{split}
\end{equation}
The quantum semigroup structure on $\QMap\bigl(\qs(M_2)\bigr)$ is given by $\DM\in\Mor(\AM,\AM\tens\AM)$ acting on generators in the following way:
\begin{equation}\label{Delta}
\begin{split}
\DM(\malpha)&=
\malpha\malpha^*\tens\malpha+\mbeta\mbeta^*\tens\malpha
+\malpha\tens\mbeta
+\malpha^*\tens\mgamma
+\malpha^*\malpha\tens\mdelta+\mgamma^*\mgamma\tens\mdelta,\\
\DM(\mbeta)&=\malpha\mgamma^*\tens\malpha+\mbeta\mdelta^*\tens\malpha
+\mbeta\tens\mbeta
+\mgamma^*\tens\mgamma
+\malpha^*\mbeta\tens\mdelta+\mgamma^*\mdelta\tens\mdelta,\\
\DM(\mgamma)&=\mgamma\malpha^*\tens\malpha+\mdelta\mbeta^*\tens\malpha
+\mgamma\tens\mbeta
+\mbeta^*\tens\mgamma
+\mbeta^*\malpha\tens\mdelta+\mdelta^*\mgamma\tens\mdelta,\\
\DM(\mdelta)&=\mgamma\mgamma^*\tens\malpha+\mdelta\mdelta^*\tens\malpha
+\mdelta\tens\mbeta
+\mdelta^*\tens\mgamma
+\mbeta^*\mbeta\tens\mdelta+\mdelta^*\mdelta\tens\mdelta
\end{split}
\end{equation}
while the counit {$\eM$} maps $\malpha,\mgamma$ and $\mdelta$ to $0$ and $\mbeta$ to $1$.

The action $\PM\in\Mor(M_2,M_2\tens\AM)$ of $\QMap\bigl(\qs(M_2)\bigr)$ on $\qs(M_2)$ is given by
\begin{equation}\label{Phibaza}
\PM(n)=nn^*\tens\malpha+n\tens\mbeta+n^*\tens\mgamma+n^*n\tens\mdelta.
\end{equation}
\end{prop}

\begin{proof}
The $\cst$-algebra $M_2$ is the universal $\cst$-algebra generated by an element $n$ satisfying the relations
\[
n^2=0\qquad\text{and}\qquad{nn^*+n^*n=\I}.
\]
One can take 
\begin{equation}\label{n}
n=\begin{bmatrix}0&1\\0&0\end{bmatrix}
\end{equation}
and we will adopt this choice. We know that the $\cst$-algebra $\AM$ exists and that it is endowed with a morphism $\PM\in\Mor(M_2,M_2\tens\AM)$ which is the quantum family of all maps (\cite[Definition 3.1(2)]{qs}) 
$\qs(M_2)\to\qs(M_2)$. This map $\PM$ has some value on the generator $n$ of $M_2$. Since $\{nn^*,n,n^*,n^*n\}$ is a basis of $M_2$ we can write $\PM(n)$ in the form \eqref{Phibaza} or as
\[
\PM:M_2\ni{n}\longmapsto\begin{bmatrix}\malpha&\mbeta\\\mgamma&\mdelta\end{bmatrix}
\in{M_2(\AM)=M_2\tens\AM},
\]
where $\malpha,\mbeta,\mgamma$ and $\mdelta$ are some elements of $\AM$. The equalities
\begin{equation*}
\PM(n^*n)=\PM(n)^*\PM(n)=\begin{bmatrix}\malpha&\mbeta\\\mgamma&\mdelta\end{bmatrix}^*
\begin{bmatrix}\malpha&\mbeta\\\mgamma&\mdelta\end{bmatrix}
=\begin{bmatrix}\malpha^*&\mgamma^*\\\mbeta^*&\mdelta^*\end{bmatrix}
\begin{bmatrix}\malpha&\mbeta\\\mgamma&\mdelta\end{bmatrix}
=\begin{bmatrix}\malpha^*\malpha+\mgamma^*\mgamma&\malpha^*\mbeta+\mgamma^*\mdelta\\
\mbeta^*\malpha+\mdelta^*\mgamma&\mbeta^*\mbeta+\mdelta^*\mdelta\end{bmatrix}
\end{equation*}
and
\begin{equation*}
\PM(nn^*)=\PM(n)\PM(n)^*=\begin{bmatrix}\malpha&\mbeta\\\mgamma&\mdelta\end{bmatrix}
\begin{bmatrix}\malpha&\mbeta\\\mgamma&\mdelta\end{bmatrix}^*
=\begin{bmatrix}\malpha&\mbeta\\\mgamma&\mdelta\end{bmatrix}
\begin{bmatrix}\malpha^*&\mgamma^*\\\mbeta^*&\mdelta^*\end{bmatrix}
=\begin{bmatrix}\malpha\malpha^*+\mbeta\mbeta^*&\malpha\mgamma^*+\mbeta\mdelta^*\\
\mgamma\malpha^*+\mdelta\mbeta^*&\mgamma\mgamma^*+\mdelta\mdelta^*\end{bmatrix}
\end{equation*}
show that the relation $nn^*+n^*n=\I_{M_2}$ implies that $\malpha,\mbeta,\mgamma$ and $\mdelta$ must satisfy \eqref{nstn}. Similarly, from
\begin{equation*}
\PM(n^2)=\PM(n)^2=\begin{bmatrix}\malpha&\mbeta\\\mgamma&\mdelta\end{bmatrix}
\begin{bmatrix}\malpha&\mbeta\\\mgamma&\mdelta\end{bmatrix}=
\begin{bmatrix}\malpha^2+\mbeta\mgamma&\malpha\mbeta+\mbeta\mdelta\\
\mgamma\malpha+\mdelta\mgamma&\mgamma\mbeta+\mdelta^2\end{bmatrix}
\end{equation*}
and the fact that $n^2=0$, it follows that relations \eqref{nkw} must be satisfied. 

Now for any $\cst$-algebra $B$ and any $\Psi\in\Mor(M_2,M_2\tens{B})$ the matrix elements of $\Psi(n)$ must satisfy the same relations as $\malpha,\mbeta,\mgamma$ and $\mdelta$. Therefore the universal $\cst$-algebra generated by $\malpha,\mbeta,\mgamma$ and $\mdelta$ with relations \eqref{nstn} and \eqref{nkw}\footnote{It follows from the form of considered relations that such a $\cst$-algebra exists} will always have a unique map $\Lambda$ onto $B$ such that $\Psi=(\id_{M_2}\tens\Lambda)\PM$. It easily follows that $\AM$ \emph{is} the universal $\cst$-algebra described in the statement of the theorem.

We know that $\AM$ possesses a comultiplication $\DM$ and that 
\begin{equation*}
(\PM\tens\id_\AM)\comp\PM=(\id_{M_2}\tens\DM)\comp\PM.
\end{equation*}
Applying both sides of this relation to $n\in{M_2}$ we obtain
\begin{equation}
\begin{split}
\begin{bmatrix}\DM(\malpha)&\DM(\mbeta)\\\DM(\mgamma)&\DM(\mdelta)\end{bmatrix}
&=(\id_{M_2}\tens\DM)\biggl(\begin{bmatrix}\malpha&\mbeta\\\mgamma&\mdelta\end{bmatrix}\biggr)
=(\id_{M_2}\tens\DM)\PM(n)=(\PM\tens\id_\AM)\PM(n)\\
&=(\PM\tens\id_\AM)
\bigl(nn^*\tens\malpha+n\tens\mbeta+n^*\tens\mgamma+n^*n\tens\mdelta\bigr)\\
&=\begin{bmatrix}(\malpha\malpha^*+\mbeta\mbeta^*)\tens\malpha
&(\malpha\mgamma^*+\mbeta\mdelta^*)\tens\malpha\\
(\mgamma\malpha^*+\mdelta\mbeta^*)\tens\malpha
&(\mgamma\mgamma^*+\mdelta\mdelta^*)\tens\malpha\end{bmatrix}+
\begin{bmatrix}\malpha\tens\mbeta&\mbeta\tens\mbeta\\
\mgamma\tens\mbeta&\mdelta\tens\mbeta\end{bmatrix}\\
&\quad+
\begin{bmatrix}\malpha^*\tens\mgamma&\mgamma^*\tens\mgamma\\
\mbeta^*\tens\mgamma&\mdelta^*\tens\mgamma\end{bmatrix}+
\begin{bmatrix}(\malpha^*\malpha+\mgamma^*\mgamma)\tens\mdelta
&(\malpha^*\mbeta+\mgamma^*\mdelta)\tens\mdelta\\
(\mbeta^*\malpha+\mdelta^*\mgamma)\tens\mdelta
&(\mbeta^*\mbeta+\mdelta^*\mdelta)\tens\mdelta\end{bmatrix}
\end{split}
\end{equation}
and we immediately find the values of $\DM$ on generators as described in \eqref{Delta}.

Finally the values of the counit $\eM$ is determined by the fact that $(\id_{M_2}\tens\eM)\PM(m)=m$ for all $m\in{M_2}$. In other words we have
\begin{equation*}
\begin{bmatrix}\eM(\malpha)&\eM(\mbeta)\\\eM(\mgamma)&\eM(\mdelta)\end{bmatrix}=
(\id_{M_2}\tens\eM)\PM(n)=n=\begin{bmatrix}0&1\\0&0\end{bmatrix}.
\end{equation*}
\end{proof}

As in Section \ref{Xn} we shall study some simple examples of quantum commutants of families of maps $\qs(M_2)\to\qs(M_2)$. The simplest possible such family is a classical family consisting of a single automorphism of $M_2$.

Let $\phi\in\Aut(M_2)$. The singleton family $\{\phi\}$ can be described in the non-commutative framework by taking $B=\CC$ and $\Psi_B:M_2\ni{m}\mapsto\phi(m)\tens{1}\in{M_2\tens{B}}$.
Now the quantum commutant of $\Psi_B$ (or in other words of $\{\phi\}$) is $(A,\Delta)$, where $A$ is the quotient of $\AM$ by the ideal \eqref{K}. In the special case we are describing one finds that this ideal is generated by
\begin{equation}\label{K2}
\Bigl\{(\omega\tens\id_\AM)\bigl(\PM(\phi(m)\bigr)-\bigl([\omega\comp\phi]\tens\id_\AM)\PM(m)\st
m\in{M_2}\;\omega\in{M_2^*}\Bigr\}.
\end{equation}
Indeed, for $m\in{M_2}$, $\eta\in{B^*}$ and $\omega\in{M_2^*}$ we have
\[
\begin{split}
(\omega\tens\id_\AM\tens\eta)(\PM\tens\id_B)\Psi_B(m)-
(\omega\tens\eta\tens\id_\AM)(\Psi_B\tens\id_\AM)\PM(m)\\
\qquad=\eta(\omega\tens\id_\AM)\PM\bigl(\phi(m)\bigr)
-\eta(\omega\tens\id_\AM)(\phi\tens\id_\AM)\PM(m).
\end{split}
\]
We can get rid of $\eta$ since it is simply a complex number and we find that the ideal \eqref{K} is generated by \eqref{K2}.

\begin{prop}\label{comalpha}
Let $\phi$ be the automorphism of $M_2$ which sends the matrix \eqref{n} to its conjugate. Then $\QMap_\phi\bigl(\qs(M_2)\bigr)=(A,\Delta)$ where $A$ is the universal $\cst$-algebra generated by three elements $\alpha,\beta$ and $\gamma$ satisfying relations 
\begin{equation}\label{comrel0}
\beta=\beta^*,\quad\gamma=\gamma^*
\end{equation}
and
\begin{subequations}\label{comrel}
\begin{align}
\alpha^*\alpha+\gamma^2+\alpha\alpha^*+\beta^2&=\I,\label{comrel1}\\
\alpha^*\beta+\gamma\alpha^*+\alpha\gamma+\beta\alpha&=0,\label{comrel2}\\
\alpha^2+\beta\gamma&=0\label{comrel3},\\
\alpha\beta+\beta\alpha^*&=0,\label{comrel4}\\
\gamma\alpha+\alpha^*\gamma&=0.\label{comrel5}
\end{align}
\end{subequations}
The comultiplication acts on generators in the following way:
\begin{equation}\label{comdel}
\begin{split}
\Delta(\alpha)&=\I_A\tens\alpha+(\alpha^*\alpha+\gamma^2)\tens(\alpha^*-\alpha)+
\alpha\tens\beta+\alpha^*\tens\gamma,\\
\Delta(\beta)&=(\alpha\gamma+\beta\alpha)\tens(\alpha-\alpha^*)+\beta\tens\beta+\gamma\tens\gamma,\\
\Delta(\gamma)&=(\beta\alpha+\alpha\gamma)\tens(\alpha^*-\alpha)+\gamma\tens\beta+\beta\tens\gamma,
\end{split}
\end{equation}
while the counit {$\varepsilon$} is
{\[
\varepsilon(\alpha)=\varepsilon(\gamma)=0,\quad{\varepsilon(\beta)=1}.
\]}
\end{prop}

\begin{proof}
Let $\alpha,\beta,\gamma$ and $\delta$ be images in $A$ of the generators $\malpha,\mbeta,\mgamma$ and $\mdelta$ of $\AM$. For the matrix $n$ given by \eqref{n} we have
\[
(\alpha\tens\id_\AM)\PM(n)-\PM\bigl(\alpha(n)\bigr)=
\begin{bmatrix}\mdelta&\mgamma\\\mbeta&\malpha\end{bmatrix}-
\begin{bmatrix}\malpha^*&\mgamma^*\\\mbeta^*&\mdelta^*\end{bmatrix}=
\begin{bmatrix}\mdelta-\malpha^*&\mgamma-\mgamma^*\\\mbeta-\mbeta^*&\malpha-\mdelta^*.\end{bmatrix}
\]
Since the image under $\pi$ of this matrix must be sent to $0$ by $\omega\tens\id_A$ for all $\omega\in{M_2^*}$ we see that the images $\alpha,\beta,\gamma$ and $\delta$ must satisfy
\begin{equation}\label{first}
\alpha=\delta^*,\quad\gamma=\gamma^*\quad\beta=\beta^*.
\end{equation}
The remaining relations \eqref{comrel} follow then directly from \eqref{nstn} and \eqref{nkw}.
Applying $\pi\tens\pi$ to both sides of \eqref{Delta} and using \eqref{first} yields at first
\[
\begin{split}
\Delta(\alpha)&=\alpha\alpha^*\tens\alpha+\beta^2\tens\alpha
+\alpha\tens\beta+\alpha^*\tens\gamma+\alpha^*\alpha\tens\alpha^*+\gamma^2\tens\alpha^*,\\
\Delta(\beta)&=\alpha\gamma\tens\alpha+\beta\alpha\tens\alpha
+\beta\tens\beta+\gamma\tens\gamma+\alpha^*\beta\tens\alpha^*+\gamma\alpha^*\tens\alpha^*,\\
\Delta(\gamma)&=\gamma\alpha^*\tens\alpha+\alpha^*\beta\tens\alpha
+\gamma\tens\beta
+\beta\tens\gamma
+\beta\alpha\tens\alpha^*+\alpha\gamma\tens\alpha^*.\\
\end{split}
\]
Then using \eqref{comrel1} \eqref{comrel2} we obtain \eqref{comdel} {and the formula for $\varepsilon$ follows from \eqref{epie}.}

So far we have not shown that $A$ is the \emph{universal} $\cst$-algebra generated by $\alpha,\beta$ and $\gamma$ satisfying \eqref{comrel0} and \eqref{comrel}. A priori some additional relations (not following from \eqref{comrel0} and \eqref{comrel}) could be satisfied in $A$. However if $\widetilde{A}$ is the universal $\cst$-algebra for considered relations then $\widetilde{\Phi}\in\Mor\bigl(M_2,M_2\tens\widetilde{A}\bigr)$ defined by
\[
\widetilde{\Phi}(n)=
\begin{bmatrix}
\alpha&\beta\\\gamma&\alpha^*
\end{bmatrix}
\]
is a quantum family commuting with $\{\phi\}$. Therefore $A=\widetilde{A}$.
\end{proof}

\begin{rem}\label{remT}
One can easily determine the \emph{classical} commutant of $\{\phi\}$. The only classical maps $\qs(M_2)\to\qs(M_2)$ {are} simply the automorphisms of $M_2$. Out of these the only ones commuting with $\phi$ are given by conjugation by matrices from the set
\[
\biggl\{
\begin{bmatrix}
a&ib\\ib&a
\end{bmatrix}
\st{a,b\in\RR,}\;a^2+b^2=1
\biggr\}\cup
\biggl\{
\begin{bmatrix}
ia&-b\\b&-ia
\end{bmatrix}
\st{b,-ia\in\RR,}\;a^2+b^2=1
\biggr\}.
\]
Therefore the classical commutant of $\phi$ is isomorphic to the group $\TT\rtimes\ZZ_2$. In particular the spectrum of $A$ must be equal to a disjoint union of two circles (cf.~\cite[Theorem 4.4 \& Corollary 4.5]{qs}).
\end{rem}

In the next corollary we will use some known facts about compact quantum groups. If $G=(B,\Delta_B)$ is such a group then the left kernel of its Haar measure $h_G$ is in fact a two sided ideal and the quotient $B_r$ of $B$ by this ideal has a canonical structure of a compact quantum group. The object $G_r=(B_r,{\Delta_B}_r)$ is called the \emph{reduced version} of $G$. There is a canonical Hopf $*$-algebra $\cB$ dense in $B$. It is the same for $G_r$ and the quotient map $B\to{B_r}$ is an identity on $\cB$. For all these results we refer the reader to \cite[page 656]{pseudogr} and \cite{cqg}.

\begin{cor}\label{corXY}
Let $\alpha$ be the automorphism of $M_2$ considered in Proposition \ref{comalpha}. Then the quantum semigroup $\QMap_{\phi}\bigl(\qs(M_2)\bigr)$ is not a quantum group.
\end{cor}

{The proof of corollary \ref{corXY} will be achieved by assuming that  $\QMap_{\phi}\bigl(\qs(M_2)\bigr)=(A,\Delta)$ is a compact quantum group and showing that the dense Hopf $*$-algebra $\cA\subset{A}$ must then contain a proper group-like projection, which is, of course, impossible. However the author is not aware of a result saying that a proper group like projection cannot be contained by $A$ (unless the Haar measure is faithful). We show that the considered projection is proper by proving that certain relations are not satisfied in $A$.}

\begin{proof}[{Proof of corollary \ref{corXY}}]
Let $\alpha,\beta$ and $\gamma$ be generators of $A$ as in Proposition \ref{comalpha} and let 
$X=\alpha+\alpha^*$, $Y=(\beta+\gamma)$. We have $X=X^*$, $Y=Y^*$. Moreover 
\[
\begin{split}
XY&=(\alpha+\alpha^*)(\beta+\gamma)=\alpha\beta+\alpha\gamma+\alpha^*\beta+\alpha^*\gamma,\\
YX&=(\beta+\gamma)(\alpha+\alpha^*)=\beta\alpha+\beta\alpha^*+\gamma\alpha+\gamma\alpha^*,
\end{split}
\]
so $XY+YX=0$ by \eqref{comrel2}, \eqref{comrel4} and \eqref{comrel5}. We will now assume that $(A,\Delta)$ is a compact quantum group and {shall} arrive at a contradiction. 

Therefore let $h$ be the Haar measure of $(A,\Delta)$. From the formulas for $\Delta$ we find that
\begin{subequations}\label{Delrel}
\begin{align}
\Delta(X)&=\I_A\tens{X}+X\tens{Y},\label{Delrel1}\\
\Delta(Y)&=Y\tens{Y},\label{Delrel2}
\end{align}
\end{subequations}
so by applying $(h\tens\id_A)\Delta$ to $X$ we obtain
\begin{equation}\label{klucz}
X+h(X)Y=\I_A.
\end{equation}
If $h(X)=0$ then $X=\I_A$, but $h(X)=h(\I_A)$ must be equal to $1$. Therefore $h(X)\neq{0}$. But then \eqref{klucz} shows that $X$ commutes with $Y$. Since they also anticommute, we see that
\[
XY=0.
\]
Now note, that 
\[
X^2+Y^2=\alpha^2+\alpha^*\alpha+\alpha\alpha^*+{\alpha^*}^2+\beta^2+\beta\gamma+\gamma\beta+\gamma^2
=\I_A
\]
by \eqref{comrel1}, \eqref{comrel3} and its adjoint version.

Therefore
\[
Y=\I_A{Y}=(X^2+Y^2)Y=X(XY)+Y^3=Y^3,
\]
so that $Y^2$ is a projection. By \eqref{Delrel2} the element $Y^2$ is also group-like. We must show that this is impossible, so we must first exclude the possibilities
\begin{enumerate}
\item\label{pierwsza} $Y^2=\I_A$,
\item\label{druga} $Y^2=0$.
\end{enumerate}
Assume that \eqref{pierwsza} holds. Then $X^2=0$, so $X=0$ and $\alpha=-\alpha^*$. Substituting this into \eqref{comrel} shows that the $\cst$-algebra $A$ is commutative, generated by three self-adjoint elements $\beta,\gamma$ and $t=i\alpha$ such that
\[
2t^2+\beta^2+\gamma^2=\I_A,\quad\text{and}\quad{t^2}=\beta\gamma.
\]
Therefore {$A=\C(Z)$}, where {$Z$} is contained in the subset of $\RR^2$ consisting of points $\left[\begin{smallmatrix}\xi\\\eta\end{smallmatrix}\right]$ such that $(\xi+\eta)=1$ and $\xi\eta\geq{0}$.\footnote{We do not investigate this issue here, but if there are any more relations between $\beta,\gamma$ and $t$, then {$Z$} is a proper subset of 
$\bigl\{\left[\begin{smallmatrix}\xi\\\eta\end{smallmatrix}\right]\st(\xi+\eta)=1,
\;\xi\eta\geq{0}\bigr\}$}
But in this case the spectrum of $A$ is too small to contain the classical commutant of $\{\phi\}$ which is homeomorphic to two disjoint circles (see Remark \ref{remT}).

The possibility \eqref{druga} is eliminated by noticing that $\beta=-\gamma$ contradicts the fact that
{\[
\varepsilon(\gamma)=0,\quad{\varepsilon}(\beta)=1,
\]}
(where {$\varepsilon$} is the counit of $(A,\Delta)$).

Thus we have established that $Y^2$ is a proper projection and a group-like element. Unlike in the theory of Hopf algebras, this fact does not, in principle, disqualify $(A,\Delta)$ as a compact quantum group. However, by \cite[Theorem 2.6(2)]{cqg} we know that if $h$ were faithful then $Y$ would have to belong to the dense Hopf $*$-algebra $\cA$ inside $A$ which is impossible. Therefore $Y$ must be sent to $0$ by the reducing map $A\to{A_r}$. Therefore in $A_r$ we have $Y=0$. We want to show that this is impossible. We cannot use the argument with counit, because $A_r$ might not have a counit. Therefore we must argue otherwise. If $Y=0$ then not only $X^2=\I_{A_r}$, but also $X=\I_{A_r}$ because of \eqref{klucz} (or \eqref{Delrel1}). Therefore $A_r$ is generated by two elements $\alpha$ and $\beta$ satisfying
\begin{subequations}
\begin{align}
\beta&=\beta^*,\label{ost1}\\
\alpha+\alpha^*&=\I_{A_r},\label{ost2}\\
\alpha^*\alpha+2\beta^2+\alpha\alpha^*&=\I_{A_r},\label{ost3}\\
\alpha^2&=\beta^2,\label{ost4}\\
\alpha\beta+\beta\alpha^*&=0,\nonumber\\
\beta\alpha+\alpha^*\beta&=0.\nonumber
\end{align}
\end{subequations}
Note first that it follows from \eqref{ost2} that $\alpha$ commutes with $\alpha^*$. Thus \eqref{ost3} can be rewritten as
\begin{equation}\label{ab}
\alpha^*\alpha=\tfrac{1}{2}(\I_{A_r}-2\beta^2).
\end{equation}
Moreover, from \eqref{ost2} we also get
\[
\alpha=\I_{A_r}\alpha=(\alpha+\alpha^*)\alpha=\alpha^2+\alpha^*\alpha
\]
so by \eqref{ost4} and \eqref{ab} we have
\[
\alpha=\beta^2+\tfrac{1}{2}(\I_{A_r}-2\beta^2).
\]
Therefore $\alpha$ commutes with $\beta$ and by \eqref{ost1} it is self-adjoint. It follows that $A_r$ is commutative. 

However, if $A_r$ is commutative then $A_r=A$ and, in particular, $A_r$ does have a continuous counit. All this shows that $Y^2$ cannot be mapped to zero by the epimorphism $A\to{A_r}$. It is therefore a non zero element of $A_r$ and by \cite[Theorem 2.6(2)]{cqg} an element of the dense Hopf $*$-algebra $\cA$ associated to $(A,\Delta)$. This, on the other hand, is impossible since $Y^2$ is proper projection and a group-like element.
\end{proof}

\end{document}